\documentclass[12pt,a4paper,twoside]{amsart}
\usepackage[utf8]{inputenc}
\usepackage[T1]{fontenc}
\usepackage{eucal}
\usepackage[english]{babel}
\usepackage{amssymb}
\usepackage{amsmath}
\usepackage{amsthm}

\title{Diophantine equations in moderately many variables}
\author{Oscar Marmon}
\date{}
\address{Department of Mathematical Sciences \\ University of Copenhagen\\ Universitetsparken 5 \\ 2100 København Ø \\Denmark}

\DeclareMathOperator{\Proj}{Proj}

\DeclareMathOperator{\Sing}{Sing}
\DeclareMathOperator{\Rad}{Rad}

\DeclareMathOperator{\rank}{rank}

\begin{document}

\newcommand{\epsi}{\varepsilon}
\newcommand{\xx}{\mathbf{x}}
\newcommand{\xxi}{\boldsymbol{\xi}}
\newcommand{\eeta}{\boldsymbol{\eta}}
\newcommand{\0}{\boldsymbol{0}}
\newcommand{\bb}{\mathbf{b}}
\newcommand{\uu}{\mathbf{u}}
\newcommand{\ee}{\mathbf{e}}
\newcommand{\vv}{\mathbf{v}}
\newcommand{\yy}{\mathbf{y}}
\newcommand{\zz}{\mathbf{z}}
\newcommand{\ww}{\mathbf{w}}
\newcommand{\ff}{\mathbf{f}}
\newcommand{\cc}{\mathbf{c}}
\newcommand{\dd}{\mathbf{d}}
\newcommand{\hh}{\mathbf{h}}
\newcommand{\ttt}{\mathbf{t}}
\renewcommand{\aa}{\mathbf{a}}
\newcommand{\bK}{\mathbf{K}}
\newcommand{\bB}{\mathbf{B}}
\newcommand{\ZZ}{\mathbb{Z}}
\newcommand{\Zpol}{\ZZ[X_1,\ldots,X_n]}
\newcommand{\FF}{\mathbb{F}}
\newcommand{\RR}{\mathbb{R}}
\newcommand{\NN}{\mathbb{N}}
\newcommand{\CC}{\mathbb{C}}
\renewcommand{\AA}{\mathbb{A}}
\newcommand{\PP}{\mathbb{P}}
\newcommand{\GG}{\mathbb{G}}
\newcommand{\QQ}{\mathbb{Q}}
\newcommand{\sB}{\mathsf{B}}
\newcommand{\cP}{\mathcal{P}}
\newcommand{\cR}{\mathcal{R}}
\newcommand{\cB}{\mathcal{B}}
\newcommand{\cE}{\mathcal{E}}
\newcommand{\cC}{\mathcal{C}}
\newcommand{\cN}{\mathcal{N}}
\newcommand{\cM}{\mathcal{M}}
\newcommand{\cL}{\mathcal{L}}
\newcommand{\cD}{\mathcal{D}}
\newcommand{\cA}{\mathcal{A}}
\newcommand{\cK}{\mathcal{K}}
\newcommand{\cF}{\mathcal{F}}
\newcommand{\cZ}{\mathcal{Z}}
\newcommand{\cX}{\mathcal{X}}
\newcommand{\cG}{\mathcal{G}}
\newcommand{\cH}{\mathcal{H}}
\newcommand{\cI}{\mathcal{I}}
\newcommand{\cO}{\mathcal{O}}
\newcommand{\cW}{\mathcal{W}}
\newcommand{\ud}{\mathrm{ud}}
\newcommand{\Zar}{\mathrm{Zar}}
\newcommand{\fm}{\mathfrak{m}}
\newcommand{\fp}{\mathfrak{p}}
\newcommand{\smod}[1]{\,(#1)}
\newcommand{\spmod}[1]{\,(\mathrm{mod}\,{#1})}
\newcommand{\bm}{\mathbf}

\renewcommand{\bf}{\mathbf}
\newcommand{\f}{\mathbf{f}}
\newcommand{\g}{\mathbf{g}}
\newcommand{\F}{\mathbf{F}}
\newcommand{\D}{\mathcal{D}}
\newcommand{\G}{\mathbf{G}}
\newcommand{\p}{\mathbf{p}}
\newcommand{\q}{\mathbf{q}}
\renewcommand{\d}{\mathbf{d}}
\newcommand{\Q}{\mathcal{Q}}
\newcommand{\llambda}{\boldsymbol\lambda}
\newcommand{\LLambda}{\boldsymbol\Lambda}

\newtheorem{thm}{Theorem}[section]
\newtheorem{lemma}[thm]{Lemma}
\newtheorem*{lemma*}{Lemma}
\newtheorem{prop}[thm]{Proposition}
\newtheorem*{prop*}{Proposition}
\newtheorem*{thm*}{Theorem}
\newtheorem{claim}[thm]{Claim}
\newtheorem{cor}[thm]{Corollary}
\newtheorem*{conj*}{Conjecture}
\newtheorem{conj}{Conjecture}
\newtheorem*{sats*}{Sats}
\theoremstyle{remark}
\newtheorem*{note*}{Note}
\newtheorem{note}{Note}
\newtheorem*{rem*}{Remark}
\newtheorem{rem}[thm]{Remark}
\newtheorem{example}[thm]{Example}
\newtheorem*{acknowledgement*}{Acknowledgement}
\newtheorem*{question*}{Question}
\newtheorem*{answer*}{Answer}
\theoremstyle{definition}
\newtheorem*{def*}{Definition}
\newtheorem{notation}[thm]{Notation}
\newtheorem*{notation*}{Notation}

\maketitle

\section{Introduction}

Let $f_1,\dotsc, f_r \in \ZZ[x_1,\dotsc,x_n]$ be polynomials of degree $d_1,\dotsc,d_r$, respectively, and let $\bf f$ denote the $r$-tuple of polynomials $(f_1,\dotsc,f_r)$. We are interested in upper bounds for the counting function
\[
N(\bf f,B) := \#\{\xx \in \ZZ^n ; f_1(\xx)=\dotsb=f_r(\xx)=0, \vert \xx \vert \leq B\}.
\]
(Here, and throughout the paper, $|\cdot|$ denotes the maximum norm $|\xx| = \max\{|x_1|,\dotsc,|x_n|\}$.) If we assume that the polynomials $f_i$ define a complete intersection in $\AA^n$ of dimension $n-r \geq 0$, then we have the well-known upper bound $N(\f,B) \ll B^{n-r}$, which we shall refer to as the trivial bound (\emph{cf.} Lemma \ref{lem:trivial} below). Heuristic arguments suggest a bound $N(\f,B) \ll B^{n-\D}$, where
\[
\D := \sum_{i=1}^r d_i,
\]
at least as soon as $n > \D$. In the special case where the polynomials $f_i$ are homogeneous of the same degree $d$, a famous result by Birch establishes the heuristic upper bound, and indeed an asymptotic formula, as soon as
\begin{equation}
\label{eq:birch}
n > s^* + 2^{d-1}(d-1)r(r+1).
\end{equation}
Here,  $s^* = s^*_\f$ is the dimension of the so-called Birch singular locus: the affine variety
\[
\{\xx \in \AA^n \mid \rank J(\xx) < r\},
\]
where $J(\xx)$ is the Jacobian matrix of size $r\times n$ with rows formed by the gradient vectors $\nabla f_i(\xx)$. (See also recent work by Dietmann \cite{Dietmann15} and, independently, Schindler \cite{Schindler15}, where $s^*$ is replaced by an alternative quantity, sometimes leading to a stronger result.)  Birch's results have recently been extended to forms of differing degree by Browning and Heath-Brown \cite{Browning-Heath-Brown14}.

Seeing as Birch's theorem, like most results proven with the Hardy-Little\-wood circle method, requires the number of variables to be rather large, one may ask if more modest upper bounds are still available for smaller values of $n$. At the far end of the spectrum, the \emph{dimension growth conjecture} of Heath-Brown and Serre leads us to expect the bound $N(\f,B) \ll B^{n-\rho-1+\epsi}$ for an $r$-tuple of homogeneous polynomials $\f$ defining an irreducible non-linear variety of codimension $\rho \leq n-2$ in $\PP^{n-1}$. The determinant method has proved a useful tool in approaching this conjecture, and it has now been established in many cases, see \cite[\S 3]{Browning09} for an overview. A full proof has recently been announced by Salberger.    

We shall allow ourselves to call an $r$-tuple of polynomials $\f$ as above a \emph{system} of polynomials, and $\d:=(d_1,\dotsc,d_r)$ its \emph{multidegree}. For each polynomial $f_i$, we denote its leading form (homogeneous part of degree $d_i$) by $F_i$. The $r$-tuple $\F:=(F_1,\dotsc,F_r)$ will then be called the \emph{system of leading forms} of $\f$. We associate to $\f$ the projective variety $Z_\f \subset \PP^{n-1}_\QQ$ defined by the leading forms $F_1,\dotsc,F_r$. By a slight abuse of notation, we write $\langle \f \rangle := \langle f_1,\dotsc,f_r \rangle$ for the ideal of $\ZZ[x_1,\dotsc,x_n]$ generated by the polynomials $f_i$. Finally, we define the \emph{height} $\Vert \f \Vert$ to be the maximal absolute value of any coefficient appearing in one of the polynomials $f_i$ in the system.  

The results in this paper, which we are now ready to state, occupy a middle ground between the two types of bounds discussed above. 
We shall first state a simplified version of our main result.

\begin{thm}
\label{thm:same}
Suppose that $d_i \geq d \geq 4$ for all $i$ and that $Z_\ff$ is non-singular of codimension $r$. Put 
\[
\eta_{n,r,d} := \frac{2^{d-2}(d-1)r}{n+2^{d-2}(d-1)r-1}.                                                                                                
\]
Then the estimate
\[
N(\ff,B) \ll_{n,\d} B^{n-rd\left(1-\eta_{n,r,d}\right)} \log(\Vert \F \Vert)^{rd}
\]
holds as soon as $n > 2^{d-2}(d-1)r$.
\end{thm}

\begin{rem*}
For any ideal $I \subset \ZZ[x_1,\dotsc,x_n]$, one may define the \emph{ideal of leading forms} $I'$, consisting of the leading forms of all elements of $I$. If $I = \langle \f \rangle$ as above, it may not always be the case that $I' = \langle \F \rangle$. However it is well-known that if the leading forms $F_i$ cut out a subscheme of codimension $r$, then they do indeed generate the ideal $I'$ of leading forms. 
\end{rem*}

\begin{rem*}
In the excluded case where $\min\{d_i\} =2$ or $3$, the estimate in Theorem \ref{thm:same} may be replaced by results of Luo \cite[Thm. 2]{Luo} or the author \cite[Thm 1.1]{Marmon08}, respectively. 
\end{rem*}

In the case where the polynomials are truly of different degree, one can do better than Theorem \ref{thm:same}. In order to state this more general result, we shall group the polynomials according to their degree. Thus, for each $2\leq d \leq D$, where $D = \max_i\{d_i\}$, let $f_{d,1},\dotsc,f_{d,r_d}$ be an enumeration of the polynomials of degree $d$ among the $f_i$. 

\begin{thm}
\label{thm:different}
Suppose that $D \geq 4$ and that $Z_\ff$ is non-singular of codimension $r$. Put
\begin{equation*}
\begin{split}
\D'&:= \sum_{d=2}^{D-1} (d-1) r_d + D r_D, \\
\Delta &:=\sum_{d=2}^{D-1} (d-2+2^{-d+1})r_d + (D-1)r_D \\
\text{and}\quad 
\eta&:= \frac{2^{D-2}\Delta}{n+2^{D-2} \Delta -1}.
\end{split}
\end{equation*}
Then, provided
\(
n > 2^{D-2} \Delta, 
\)
we have the estimate
\begin{equation}
\label{eq:different}
N(\f,B) \ll_{n,\d} B^{n-\D'\left(1- \eta \right)} \log(\Vert \F \Vert)^\cD. 
\end{equation}
\end{thm}

To compare the admissible range of $n$ in the above theorems with \eqref{eq:birch}, we consider the singular locus, in the usual sense, of the projective variety $Z_\f$, which is the variety
\[
\Sing(Z_\f) = \{\xx \in \PP^{n-1}_\QQ \mid F_1(\xx) = \dotsb = F_r(\xx) = 0, \rank J(\xx) < r\}.
\]
Here, $J(\xx)$ is the Jacobian matrix with rows $\nabla F_i(\xx)$. Putting $s = \dim (\Sing(Z_\f))$, one has
\[
s+1 \leq s^* \leq s+1 +r,
\]
for $s^* = s^*_\F$. The non-singularity assumption in our theorems translates to the condition that $s=-1$, but it would be an easy matter to derive more general results, with $s+1$ playing a similar role as the quantity $s^*$ in \eqref{eq:birch}.

The proofs of Theorems \ref{thm:same} and \ref{thm:different} employ an iterated version of the multidimensional $q$-analogue of van der Corput differencing introduced by Heath-Brown in \cite{Heath-Brown}, and also builds upon previous work by the author \cite{Marmon08,Marmon10}.

\begin{rem}
\label{rem:improvement}
We observe that in the case where the polynomials are of equal degree, the exponent of $B$ in Theorem \ref{thm:same} may be written as $n-\mu \cD$, where $\mu \to 1$ as $n \to \infty$. Thus our bound approaches the heuristic one asymptotically in this sense. The fact that the same is not true, in general, of the exponent in Theorem \ref{thm:different}, indicates that our method is open to further improvement.
\end{rem}

We shall now describe the strategy behind the proof of Theorems \ref{thm:same} and  \ref{thm:different}. Let us first introduce some notation that will be used. 

\begin{notation*}
We define 
\[
\cZ_\f := \Proj \ZZ[x_1,\dotsc,x_n]/\langle \F \rangle \subset \PP^{n-1}_\ZZ,
\]
so that the variety $Z_\f$ featuring in our main results may be expressed as
\[
Z_\f = \cZ_\f \otimes_\ZZ \QQ. 
\]
For any prime $p$, reduction modulo $p$ gives rise to the projective variety
\[
Z_{\f,p} := \cZ_\f \otimes_{\ZZ} \FF_p = \Proj \ZZ[x_1,\dotsc,x_n]/\langle p,F_1,\dotsc,F_r \rangle \subset \PP^{n-1}_{\FF_p}.
\]
Writing $Z_{\f,\infty} := Z_\f$ and $\FF_\infty := \QQ$ for notational convenience, we now define
\[
\rho_v(\f) := \operatorname{codim}(Z_{\f,v},\PP^{n-1}_{\FF_v}) \quad \text{and} \quad 
s_v(\f) := \dim (\Sing Z_{\f,v}), 
\]
where either $v=p$ for a prime $p$, or $v=\infty$. The parameters $\rho_v(\f)$ and $s_v(\f)$ will appear frequently in our arguments. They clearly satisfy the inequalities 
\[
0 \leq \rho_v(\f) \leq r, \qquad 
-1 \leq s_v(\f) \leq n - 1- \rho_v(\f). 
\]
It is also clear that the inequality $\rho_p(\f) \leq \rho_\infty(\f)$ holds for all $p$, and that $\rho_p(\f) = \rho_\infty(\f)$ for $p \gg_\f 1$. Conversely, we have $s_p(\f) \geq s_\infty(\f)$, with equality for $p \gg_\f 1$.
\end{notation*}

Our approach for estimating $N(\f,B)$ involves replacing the system of equations $f_{d,1}(\xx) = \dotsb = f_{d,r_d}(\xx) = 0$, for each $2 \leq d \leq D$, with a system of congruences
\begin{equation}
\label{eq:congruences}
f_{d,1}(\xx) \equiv \dotsb \equiv f_{d,r_d}(\xx) \equiv 0 \pmod{q_d}
\end{equation}
for suitably chosen integers $q_2,\dotsc,q_D$. More precisely, let $m$ be an integer between $0$ and $D-2$. Then, given a parameter $\xi \in [B^{1/2},B]$,  and a collection of pairwise different primes $p_0,\dotsc,p_m$, satisfying
\begin{gather}
\label{eq:primes_size}
p_1 \asymp \dotsb \asymp p_{m} \asymp \xi, \quad p_0 \asymp \xi^2 \text{ and} \\
\label{eq:primes_gcd}
(p_0 p_1 \dotsb p_m,D!) = 1,
\end{gather}
we take
\begin{equation}
q_d := \prod_{i=0}^{\min\{m,d-2\}} p_{m-i}.
\end{equation}
\begin{rem*}
The natural number $m$ represents the number of differencing steps that will be performed on the polynomials of highest degree $D$. Each congruence to the modulus $q_d$ will correspond, after one step in our van der Corput differencing process, to a family of congruences of degree $d-1$ to the smaller modulus $p_m^{-1} q_d$. The numbers $q_d$ are chosen in such a way that we only ever need to consider congruences involving polynomials of degree at least $2$ (discarding, as we may, all congruences to the modulus $1$).  
\end{rem*}

To encode the above congruence relations, we now put
$$\q_d:=(q_d,\dotsc,q_d) \in \NN^{r_d}$$
for $2 \leq d \leq D$ and let $\q \in \NN^r$ be the vector defined by
\begin{equation}
\q = (q_2,\dotsc,q_2,\dotsc,q_D,\dotsc,q_D)= (\q_2,\dotsc,\q_D).
\end{equation}

Now, for a suitably chosen smooth weight function $W \in C_0^\infty(\RR^n)$, we certainly have
\begin{equation}
\label{eq:N(f,B,q)}
N(\f,B) \leq \sum_{\substack{\xx \in \ZZ^n\cap [-B,B]^n\\ \q \mid \f(\xx)}} 1 \ll \sum_{\substack{\xx \in \ZZ^n\\ \q \mid \f(\xx)}} W(B^{-1}\xx) =:N_W(\f,B,\q),
\end{equation}
where the condition $\q \mid \f(\xx)$ should be interpreted to mean that the congruence systems \eqref{eq:congruences} are simultaneously satisfied for $2 \leq d \leq D$. 

In estimating the quantity $N_W(\f,B,\q)$, we shall group the polynomials $f_i$ into blocks in two different ways --- we put
\[
\f_d = (f_{d,1},\dotsc,f_{d,r_d}) \in \ZZ[\xx]^{r_d}
\]
for each $2 \leq d \leq D$, and furthermore
\[
\hat \f_i = (\f_{i+2},\f_{i+3},\dotsc,\f_{D})
\in \ZZ[\xx]^{r_{i+2}+\dotsb + r_D}
\]
for each $0 \leq i \leq m$. (Thus, each system $\hat \f_i$ is composed of all polynomials $f_{d,j}$ for which $p_i \mid q_d$, so that they will be subject to at least $i$ differencing steps. In particular, $\hat \f_0 = \f$ and $\hat \f_m = (\f_{m+2},\dotsc,\f_D)$.) 

We are now ready to formulate the main technical result of the paper, giving an asymptotic formula for a weighted counting function $N_W(\f,B,\q)$. The class of smooth weight functions $\cC_n(R,(\kappa_j))$ occurring in the statement will be defined in Section \ref{sec:preliminary}. 

\begin{prop}
\label{prop:heart}
Let $\xi \in [B^{1/2},B]$ and suppose that $p_0,\dotsc,p_m$ are primes satisfying \eqref{eq:primes_size} and \eqref{eq:primes_gcd}. Suppose that $\rho_{p_i}(\f) = r$ for all $0 \leq i \leq m$, and put
\[
s:= \max_{0 \leq i \leq m} s_{p_i}(\hat \f_i). 
\]
Then, for any smooth weight function $W \in \cC_n(R,(\kappa_j))$, the asymptotic formula
\begin{equation}
\label{eq:asymptotic'}
\begin{split}
N_W(\f,B,\q) - \frac{\cB}{q_2^{r_2}q_3^{r_3}\dotsb q_D^{r_D}}  &\ll B^n \xi^{-\cR} \left(\frac{\xi}{B}\right)^{(n-s-2)/2^m} \\
&\hphantom{\ll} + B^n \xi^{-r/2} \left(\frac{\xi}{B} \right)^{(n-s-1)/2} 
\end{split}
\end{equation}
holds, where
\[
\cR = \sum_{i=2}^{m+1} (1-2^{-i+1})r_{i} + \sum_{i = m+2}^D r_i
\]
and
\[
\cB = \cB_W :=  \sum_{\xx \in \ZZ^n} W(B^{-1}\xx).
\]
The implied constant in \eqref{eq:asymptotic'} depends at most on the data 
\[
n, \d, m, R \text{ and }(\kappa_j).
\]
\end{prop}

The case $m=0$ in Proposition \ref{prop:heart} gives an asymptotic formula for the weighted number of solutions of height at most $B$ to a system of congruences modulo a single prime $p_0 \in [B,B^2]$, with an error term depending on the dimension of the singular locus of the corresponding variety. Such a result appeared as Theorem 3.3 in \cite{Marmon08}, Heath-Brown having treated the case $r=1$ in \cite{Heath-Brown}. It may be viewed as an extension of Hooley's estimate \cite{Hooley} for the number of points on a complete intersection over a finite field, which in turn generalized Deligne's estimate \cite[Thm 8.1]{Deligne_WeilI}.

The case $m=1$, where a single differencing step is performed, was treated in \cite{Marmon08}. To deduce Theorem \ref{thm:same} from Proposition \ref{prop:heart}, we shall take $m = d-2 \geq 2$, thus performing the same number of differencings on each of the polynomials. For Theorem \ref{thm:different}, we instead take $m = D-2$, performing the maximum number of differencings on the polynomials of highest degree, and gradually fewer for those of lower degree.

After collecting the necessary tools in Section \ref{sec:preliminary}, we devote Section \ref{sec:proof_heart} to the proof of Proposition \ref{prop:heart}. The deduction of Theorems \ref{thm:same} and \ref{thm:different} is carried out in Section \ref{sec:proof_main}. 
\section{Preliminary considerations}
\label{sec:preliminary}

The differencing process that will be used in Section \ref{sec:proof_heart} gives rise to new polynomials
\[
f_i^\yy(\xx) = f_i(\xx+p_m \yy) - f_i(\xx).
\]
If $\yy \neq \0$, the polynomial $f_i^\yy$ has degree $d_i -1$, and its leading form is
\[
F_{i}^{\mathbf{y}}(\xx) = p_m\mathbf{y}\cdot\nabla F_{i}(\xx).
\] 
To apply Proposition \ref{prop:heart} iteratively, we need to control the dimension of singular loci of varieties defined by collections of such forms. The following result, which appeared in previous work of the author \cite[Lemma 2.9]{Marmon08}, addresses that problem. However, that version contained two errors, which we now take the opportunity to correct:
\begin{itemize}
\item The condition $p \nmid d_i$ in the hypotheses needs to be strengthened to $p \nmid d_i(d_i-1)$ in order for the proof to be valid.
\item In the proof of item (i), the identity $\xx \cdot \nabla^2G_i(\xx) = \ldots$
is both erroneous and superfluous, and one gets a valid proof by simply ignoring this statement.
\end{itemize}

\begin{lemma}
\label{lem:HBlemma2}
Let $G_{1},\ldots,G_{r} \in  \mathbb{Z}[x_{1},\ldots,x_{n}]$ be homogeneous polynomials of degrees $d_{1},\ldots,d_{r}$, respectively,
let $p$ be a prime such that $p\nmid d_{i}(d_{i}-1)$ for all $i=1,\ldots,r$,
and suppose that $\rho_p(\G) = r$ and $s_p(\G) = -1$. Define closed subschemes $S \subseteq \mathbb{P}_{\mathbb{F}_{p}}^{n-1}\times\mathbb{P}_{\mathbb{F}_{p}}^{n-1}$ and $S_\yy \subseteq \mathbb{P}_{\mathbb{F}_{p}}^{n-1}$ for each $\yy \in \PP^{n-1}_{\FF_p}$ by setting
\begin{multline*}
S=\left\{ (\mathbf{x},\mathbf{y})\in\mathbb{P}_{\mathbb{F}_{p}}^{n-1}\times\mathbb{P}_{\mathbb{F}_{p}}^{n-1};\ \mathbf{y}\cdot\nabla G_{i}(\mathbf{x})=0,\ i=1,\ldots,r,\right.\\
\left.\mathrm{rank}\left(\mathbf{y}\cdot\nabla^{2}G_{i}(\mathbf{x})\right)_{1\leq i\leq r}<r\right\}
\end{multline*}
and
\begin{multline*}
S_{\mathbf{y}}=\left\{ \mathbf{x}\in\mathbb{P}_{\mathbb{F}_{p}}^{n-1};\ \mathbf{y}\cdot\nabla G_{i}(\mathbf{x})=0,\ i=1,\ldots,r,\right.\\
\left.\mathrm{rank}\left(\mathbf{y}\cdot\nabla^{2}G_{i}(\mathbf{x})\right)_{1\leq i\leq r}<r\right\}.
\end{multline*}
For $s=-1,0,1,\ldots,n-1,$ let $T_{s}=\left\{ \mathbf{y}\in\mathbb{P}_{\mathbb{F}_{p}}^{n-1};\ \dim S_{\mathbf{y}}\geq s\right\}.$ Then $T_s$ is also Zariski closed, and the following estimates hold:
\begin{itemize}
\item [(i)]\label{enu:dimS}
$\dim S\leq n-2$;
\item [(ii)]\label{enu:dimT_s}
$\dim T_{s}\leq n-s-2$;
\item [(iii)]
If $T_{s}^{(1)},T_{s}^{(2)},\ldots$ are the
irreducible components of $T_{s}$, then
\[
\sum_{j}\deg(T_{s}^{(j)})=O_{n,\d}(1).
\]
\end{itemize}
\end{lemma}

\begin{rem}
\label{rem:degenerate}
In particular, it follows that the set of $\yy$ such that $\yy \cdot G_i(\xx)$ vanishes identically in $\xx$ for some $i$ is (projectively) empty. Indeed, in this case we would have $S_\yy = \PP^{n-1}$, and by (\ref{enu:dimT_s}), we have $T_{n-1}= \emptyset$. 
\end{rem}

In our arguments we shall use smooth weight functions, for which we introduce the following notation, partly following \cite{Heath-Brown96}.

\begin{def*}[Smooth weights]
For any function $W \in C_0^\infty(\RR^n)$, we define its radius $\Rad(W)$ to be the smallest real number $R$ such that $\operatorname{supp}(W) \subset [-R,R]^n$. Furthermore, we put
\[
\kappa_j(W) := \max_{j_1+\dotsb+j_n = j} \max_{\xx \in \RR^n} \left\vert \frac{\partial^j}{\partial x_1^{j_1}\dotsb \partial x_n^{j_n}}\,W(\xx)  \right\vert
\]
for any integer $j \geq 0$. In particular, $\kappa_0(W) = \max_\xx \vert W(\xx) \vert$. Using this notation, given a positive real number $R$ and a sequence of positive real numbers $(\kappa_j)_{j=0}^\infty$, we let $\cC_n(R,(\kappa_j))$ be the set of all functions $W \in C_0^\infty(\RR^n)$ such that 
\[
\Rad(W) \leq R \quad \text{and} \quad \kappa_j(W) \leq \kappa_j \text{ for all } j \in \ZZ_{\geq 0}.
\]
\end{def*}
In our next result, we list some elementary properties of the classes $\cC_n(R,(\kappa_j))$.

\begin{prop}
\label{prop:weights}
Suppose that $W \in \cC_n(R,(\kappa_j))$. Then the following properties hold:
\begin{itemize}
\item[(i)]
For any $\uu \in \RR^n$, the translation $\tau_\uu W$ defined by $\xx \mapsto W(\xx+\uu)$ satisfies $\tau_\uu W \in \cC_n(R+|\uu|,(\kappa_j))$.
\item[(ii)]
For any $x \in \RR$, then the restriction $r_x W:\RR^{n-1}\to \RR$ defined by $\xx \mapsto W(\xx,x)$ satisfies $r_x W \in \cC_{n-1}(R,(\kappa_j))$.
\item[(iii)]
If $W' \in \cC_n(R',(\kappa_j'))$, then $WW' \in \cC_n(R'',(\kappa_j''))$ where
\[
R'' = \min\{R,R'\} \quad \text{and} \quad \kappa_j'' \ll_n \sum_{j_1+j_2=j} \kappa_{j_1} \kappa'_{j_2}.
\]
\item[(iv)]
For any matrix $M=(m_{ij}) \in \operatorname{GL}_n(\RR)$, the composition $W \circ M$ satisfies $W \circ M \in \cC(R',(\kappa_j'))$, where
\[
R' \ll_{n,\Vert M^{-1}\Vert} R \quad \text{and} \quad \kappa_j' \ll_{n,\Vert M \Vert} \kappa_j,                                                                                                                                                                                                                                                                                 \]
with $\Vert M \Vert := \max |m_{ij}|$.
\end{itemize}
\end{prop}

\begin{proof}
The properties (i) and (ii) follow directly from the definition above, whereas (iii) and (iv) are easy consequences of the Leibniz formula and the chain rule, respectively.
\end{proof}

Our next result is a reformulation of \cite[Lemma 3.1]{Marmon10} in the notation introduced above. It is a consequence of the Poisson summation formula and plays a  key role in the proof of Proposition \ref{prop:heart}.

\begin{lemma}
\label{lem:poisson}
Let $W \in \cC(R,(\kappa_j))$, and let $a$ and $B$ be real numbers with $B\geq1$ and $1\leq a\leq B$. Then we have
\begin{multline*}
  \sum_{\xx\in\ZZ^n} W\left(\frac{1}{B}\xx\right) \sum_{\yy\in\ZZ^n} W\left(\frac{1}{B}(\xx+a\yy)\right) \\
= a^{-n}\left(\sum_{\xx\in\ZZ^n} W\left(\frac{1}{B}\xx\right)\right)^2 + O_{n,R,N}\left(\kappa_0 \kappa_N B^{2n-N}a^{-n+N}\right)\\
 + O_{n,R,N} \left(\kappa_N^2 B^{2(n-N)}a^{-n+N}\right)
 \end{multline*}
 for any $N \in \ZZ_{\geq 0}$.
\end{lemma}

We may view the above result as counting, with smooth weights, the number of points of height at most $B$ in the lattice $\{(\xx,\xx + a\yy) \mid \xx,\yy \in \ZZ^n\} \subseteq \RR^{2n}$. This lattice has determinant $a^n$, explaining the main term in the asymptotic formula, and the smooth weight allows for a strong error term.

Given any integer $q$, we may define a congruence counting function
\[
N(\bf f,B,q) := \#\{\xx \in \ZZ^n ; f_1(\xx)\equiv \dotsb \equiv f_r(\xx) \equiv 0 \spmod q, \vert \xx \vert \leq B\}. 
\]
The following result records 'trivial' estimates for the counting functions $N(\f,B)$ and $N(\f,B,p)$ for any prime $p$.
\begin{lemma}
\label{lem:trivial} 
Given $\f \in \ZZ[x_1,\dotsc,x_n]^r$, we have the upper bound
\begin{equation}
\label{eq:trivial_B}
N(\f,B) \ll_{n,\d} B^{n-\rho_\infty(\f)}.
\end{equation}
Furthermore, for any prime $p$, we have
\begin{equation}
\label{eq:trivial_Bp}
N(\f,B,p) \ll_{n,\d} B^{n-\rho_p(\f)} + B^n p^{-\rho_p(\f)}.
\end{equation}
In particular, 
\begin{equation} 
\label{eq:trivial_p}
\# Z_{\f,p}(\FF_p) \ll_{n,\d} p^{n-\rho_p(\f)}.
\end{equation}
\end{lemma}

Such bounds are certainly used frequently, with or without proof, in the literature. In particular, a well-known reference for the bound \eqref{eq:trivial_p} is the classic paper by Lang and Weil \cite{LangWeil}. We allow ourselves to give a self-contained proof of all of these statements, thereby reproducing parts of the proof of \cite[Lemma 4]{Browning-Heath-Brown09}.

\begin{proof}
Let us first assume that $B \leq p$ and prove the bound 
\begin{equation}
\label{eq:trivial1}
N(\f,B,p) \ll_{n,\d} B^{n-\rho_p(\f)}, 
\end{equation}
which will establish \eqref{eq:trivial_Bp} in this case. For any closed subvariety $V \subset \AA^n_{\FF_p}$, let $V(B)$ be the set of points in $V(\FF_p)$ having a representative $\xx$ in $\ZZ^n$ with $\vert \xx \vert \leq B$. Furthermore, let $\delta(V)$ be the sum of the degrees of the irreducible components of $V$. We shall prove that
\begin{equation}
\label{eq:trivial2}
\#V(B) \ll_{\delta(V)} B^{\dim(V)}. 
\end{equation}
Since $\delta(Z_{\f,p}) \leq \prod d_i$ by B\'ezout's Theorem \cite[Ex. 8.4.6]{Fulton}, this will imply the desired bound 
\eqref{eq:trivial1}. 
We proceed by induction on the dimension $m:= \dim(V)$. If $m=0$, it is clear that $\#V(B) \leq \delta(V)$. Assume now that $m \geq 1$. We may further assume that $V$ is irreducible, since it has at most $\delta(V)$ irreducible components, each of dimension at most $m$. By irreducibility, we may choose $i \in \{1,\dotsc,n\}$ such that the hyperplane $H_b$ defined by $x_i = b$ intersects $V$ properly for any $b \in \FF_p$, and write
\[
\#V(B) \leq \sum_{b \in \ZZ; |b| \leq B} \#(V\cap H_b)(B). 
\]
Now we have $\delta(V \cap H_b) \leq \delta(V) = \deg(V)$, so our induction hypothesis implies that $\#(V\cap H_b)(B) \ll_{\delta(V)} B^{m-1}$ for each $b$ in the sum, and the desired bound \eqref{eq:trivial2} follows. 

Taking $B = p$ in  \eqref{eq:trivial1}, we immediately infer the bound \eqref{eq:trivial_p}. Furthermore, we clearly obtain a proof of \eqref{eq:trivial_B} by replacing $\FF_p$ by $\QQ$ in the proof of \eqref{eq:trivial1}.

Suppose next that $B > p$. For each $\uu \in \FF_p^n$, the number of points $\xx$ counted by $N(\f,B,p)$ whose reduction (mod $p$) is $\uu$, is at most $O_n((B/p)^n)$. By \eqref{eq:trivial_p}, we conclude that
\(
N(\f,B,p) \ll_{n,\d}  B^n p^{-\rho_p(\f)},
\)
thus establishing the remaining case of \eqref{eq:trivial_Bp}.
\end{proof}

Our final result in this section is a variant of Bertini's theorem, formulated in an 'effective' way. Here, for $\aa \in \ZZ^n$, we use the notation $H_\aa \subset \PP^{n-1}$ for the hyperplane defined by the equation $\aa \cdot \xx = 0$.

\begin{lemma}
\label{lem:hyperplane}
Let $\Pi=\{p_1,\dotsc,p_N\}$ be a finite set of primes, and suppose that $\rho_{p}(\f)=r$ for each $p\in \Pi$. Then, provided  $\sum_{i=1}^N \frac{1}{p_i} \ll_{n,\d} 1$, there is a primitive integer vector $\aa \in \ZZ^n$, with $\aa \ll_{n,N,\d} 1$, such that the properties
\begin{itemize}
\item [{(i)}] $\dim(\Sing (Z_{\f,p} \cap H_{\mathbf{a}}))=\max\{-1,\dim(\Sing Z_{\f,p})-1\}$,
\item [{(ii)}]$\dim(Z_{\f,p}\cap H_{\mathbf{a}}) =\dim Z_{\f,p}-1$
\end{itemize}
hold for each $p \in \Pi$.
\end{lemma}

\begin{proof}
It follows from \cite[Lemma 2.8]{Marmon08} that there exist forms $\Phi_i \in \ZZ[\xi_1,\dotsc,\xi_n]$ of degree $O_{n,\d}(1)$, not identically divisible by $p_i$, such that the properties (i) and (ii) hold for $p = p_i$ as soon as $p_i \nmid \Phi_i(\aa)$. For any $A \in \NN$, let 
\[
\cM(A) := \{\aa \in \ZZ^n \cap [-A,A]^n;\ \gcd(a_1,\dotsc,a_n) = 1\}.  
\]
We note that
\[
\#\cM(A) = \frac{2^n}{\zeta(n)}A^n + O(A^{n-1}\log(A)) > A^n
\]
as soon as $A$ exceeds some explicit constant $C_1(n)$. Furthermore, we put
\[
\cN_i(A) := \{\aa \in \cM(A);\ p_i\mid \Phi_i(\aa)\} 
\]
for $i=1,\dotsc,n$ and $\cN(A) := \cup_i \cN_i(A)$. By Lemma \ref{lem:trivial}, there is then a constant $C_2(n,\d)$ such that 
\[
\#\cN_i(A) \leq C_2(n,\d) \left(A^{n-1} + A^n p_i^{-1} \right).
\]
for all $i$. This implies that
\[
\#\cN(A) \leq C_2(n,\d) \left(NA^{n-1} + A^n \sum_{i=1}^N \frac{1}{p_i}\right),
\]
so if $$\sum_{i=1}^N \frac{1}{p_i} \leq \frac{1}{2C_2(n,\d)} \quad \text{and} \quad A \geq \max\left\{C_1(n), \frac{2N}{C_2(n,\d)}\right\}$$ we get
$\#\cN(A) \leq A^n < \#\cM(A)$, and in particular $\cN(A) \neq \cM(A)$, establishing the existence of an integer vector $\aa$ with the desired properties.
\end{proof}


\section{Proof of Proposition \ref{prop:heart}}
\label{sec:proof_heart}
We shall argue by induction over $m$. The inductive base $m=0$ follows from \cite[Thm. 3.3]{Marmon08}. Indeed, we then have $\cR=r$ and $q_d = p_0$ for all $d$. In this case, \cite[eq.~(5)]{Marmon08}, with $q=p_0 \asymp \xi^2$, reads
\[
N_W(\f,B,\q) - \frac{\cB}{q^r} \ll B^{s+2}\xi^{n-r-s-2} = B^n \xi^{-r} \left(\frac{\xi}{B} \right)^{n-s-2},
\]
as required for \eqref{eq:asymptotic'}.

For $m \geq 1$, we shall use induction over $s$, the inductive base being the case $s=-1$. Thus, suppose that $s_{p_i}(\hat \f_i) = -1$ for all $0 \leq i \leq m$. The quantity that we wish to estimate is
\[
V_W(\f,B,\q) := N_W(\f,B,\q) - \frac{\cB}{Q},
\]
where
\(
Q:=q_2^{r_2}q_3^{r_3}\dotsb q_D^{r_D}.
\)
We shall do this by carrying out a $p_m$-van der Corput differencing step, reducing all the moduli involved by a factor $p_m$. Our first step is to split the sum into congruence classes modulo $p_m$, yielding
\begin{align}
\label{eq:splitsum}
V_W(\f,B,\q) &= \sum_{\substack{\uu \spmod{p_m}\\ p_m \mid \f(\uu)}} \sum_{\substack{\xx \equiv \uu \spmod{p_m}\\ \q \mid \tilde\f(\xx)}} W(B^{-1}\xx) - Q^{-1} \cB,
\end{align}
where $\tilde\f:= \hat\f_1 = (\f_3,\dotsc,\f_D)$.
Note that if $\xx \equiv \uu \spmod{p_m}$, where $p_m \mid \f(\uu)$, then for each $3 \leq d \leq D$, the condition $q_d \mid \f_d(\xx)$ is equivalent to $\tilde q_d \mid \f_d(\xx)$, where $\tilde q_d := p_m^{-1}q_d$.
Therefore, the 'expected value' of the inner sum 
\[
\sum_{\substack{\xx \equiv \uu \spmod{p_m}\\ \q \mid \tilde\f(\xx)}} W(B^{-1}\xx) =: \Upsilon(\uu),
\]
say, is $p_m^{-n} \tilde Q^{-1} \cB$,  where
\[
\tilde Q := \tilde{q}_3^{r_3} \dotsb \tilde q_D^{r_D} = p_m^{-r} Q.
\]
Writing $\tilde \q := (\tilde q_3,\dotsc,\tilde q_3,\dotsc,\tilde q_D,\dotsc,\tilde q_D)\in \NN^{r_3} \times \dotsb \times \NN^{r_D}$ and
\[
S := \sum_{\substack{\uu \spmod{p_m}\\ p_m \mid \f(\uu)}} \left\{ 
\Upsilon(\uu) - p_m^{-n} \tilde Q^{-1} \cB \right\},
\]
we obtain
\begin{align*}
V_W(\f,B,\q) &= S + p_m^{-n+r} Q^{-1} \cB \sum_{\substack{\uu \spmod{p_m}\\ p_m \mid \f(\uu)}} 1 - Q^{-1} \cB.
\end{align*}
We have assumed that $\rho_{p_m}(\f) = r$ and $s_{p_m}(\f) = -1$, so we have
\[
\sum_{\substack{\uu \spmod{p_m}\\ p_m \mid \f(\uu)}} 1 = p_m^{n-r} + O\left(p_m^{(n-r+1)/2}\right),
\]
by Hooley's extension \cite{Hooley} of Deligne's theorem (see \cite[Lemma 3.2]{Marmon08}).
Observing that
\(
Q \asymp \xi^{\kappa}, 
\)
where 
\[
\kappa := \sum_{i=2}^{m+1} (i-1)r_i + (m+2)\sum_{i=m+2}^D r_i,
\]
we conclude that
\begin{equation}
V_W(\f,B,\q) = S + O\left(B^n \xi^{-\kappa - (n-r-1)/2}\right).
\end{equation}
The error term here is clearly admissible for \eqref{eq:asymptotic'}. Indeed, since $\kappa - \frac{r}2 \geq \cR$ and $\xi^{-1} \leq \xi/B \leq 1$, we have
\begin{align*}
B^n \xi^{-\kappa - (n-r-1)/2} &< B^n \xi^{-\cR-(n-1)/2} \leq B^n \xi^{-\cR} \left(\frac{\xi}{B} \right)^{(n-1)/2} \\
&\leq B^n \xi^{-\cR} \left(\frac{\xi}{B} \right)^{(n-1)/2^m},
\end{align*}
as required.

Applying Cauchy's inequality to the sum defining $S$, and using \eqref{eq:trivial_p} with $\rho_{p_m}(\f) = r$, we get
\[
S \ll p_m^{(n-r)/2} \Sigma^{1/2},
\]
and thus
\[
V_W(\f,B,\q) \ll p_m^{(n-r)/2} \Sigma^{1/2} + B^n \xi^{-\cR} \left(\frac{\xi}{B} \right)^{(n-1)/2^m}
\]
where
\[
\Sigma := \sum_{\uu \spmod{p_m}} \left\{ \Upsilon(\uu) - p_m^{-n} \tilde Q^{-1} \cB \right\}^2
\]
may be viewed as the 'variance' of the quantity $\Upsilon(\uu)$ introduced above.
To facilitate the analysis of this sum, we add some extra terms to it, obtaining
\begin{equation}
\label{eq:sigma'}
\Sigma \leq
\sum_{3 \leq d \leq D} \sum_{\aa_d \in (\ZZ/\tilde q_d \ZZ)^{r_d}} \sum_{\uu \spmod{p_m}} T_\uu(\aa_3,\dotsc,\aa_D)^2,
\end{equation}
where
\[
T_\uu(\aa_3,\dotsc,\aa_D):=\sum_{\substack{\xx \equiv \uu \spmod{p_m}\\ \f_d(\xx) \equiv  \aa_d \spmod{\tilde q_d},\ 3 \leq d \leq D
}} W(B^{-1}\xx) - p_m^{-n} \tilde Q^{-1} \cB.
\]
In particular, $T_\uu(\0,\dotsc,\0) = \Upsilon(\uu) - p_m^{-n} \tilde Q^{-1} \cB$.
\begin{rem}
\label{rem:salberger}
As seen in \cite{Marmon10}, one can obtain better bounds by circumventing this step, at the cost of more complicated geometric considerations. The current approach, however, seems better suited for an inductive argument with an arbitrary number of differencing steps.
\end{rem}

When summing the expression 
\begin{multline*}
T_\uu(\aa_3,\dotsc,\aa_D)^2 = \left\{\sum_{\substack{\xx \equiv \uu \spmod{p_m}\\ \f_d(\xx) \equiv  \aa_d \spmod{\tilde q_d},\ 3 \leq d \leq D
}}  \hspace{-15pt} W(B^{-1}\xx)\right\}^2 \\
- 2 p_m^{-n} \tilde Q^{-1} \cB 
\left\{\sum_{\substack{\xx \equiv \uu \spmod{p_m}\\ \f_d(\xx) \equiv  \aa_d \spmod{\tilde q_d},\ 3 \leq d \leq D
}}  \hspace{-15pt} W(B^{-1}\xx)\right\} 
+ p_m^{-2n} \tilde Q^{-2} \cB^2
\end{multline*}
over $\aa_3,\dotsc,\aa_D$ and $\uu$ as in \eqref{eq:sigma'}, we see that the contribution from the last term precisely cancels the contribution from one of the cross terms, 
and thus the right hand side of \eqref{eq:sigma'} equals
\begin{multline*}
\hspace{-5pt}\sum_{3 \leq d \leq D} \sum_{\aa_d \spmod{\tilde q_d}} \sum_{\uu \spmod{p_m}} \left\{\sum_{\substack{\xx \equiv \uu \spmod{p_m}\\ \f_d(\xx) \equiv  \aa_d \spmod{\tilde q_d},\ 3 \leq d \leq D
}}  \hspace{-15pt} W(B^{-1}\xx)\right\}^2
- p_m^{-n} \tilde Q^{-1} \cB^2. \hspace{-10pt}
\end{multline*}
Expanding the square, we express the leftmost sum as
\begin{align*}
&\hphantom{=} \sum_{\xx \in \ZZ^n} \sum_{\substack{\xx' \equiv \xx \spmod{p_m}\\ \f_d(\xx') \equiv \f_d(\xx) \spmod{\tilde q_d},\ 3\leq d \leq D}} W(B^{-1}\xx) W(B^{-1}\xx')\\
&=\sum_{\yy \in \ZZ^n} \sum_{\substack{\xx \in \ZZ^n\\ \tilde q_d \mid \f_d(\xx+p_m \yy)-\f_d(\xx),\ 3 \leq d \leq D}} W(B^{-1}(\xx+p_m \yy)) W(B^{-1}\xx).
\end{align*}
Now we define, for each $\yy \in \ZZ^n$ and each polynomial $f_i$, the differenced 
polynomial $f_i^\yy$ by setting
\[
f_i^\yy(\xx) = f_i(\xx+p_m \yy) - f_i(\xx).
\]
Accordingly, we put $\f_d^\yy :=(f_{d,1}^\yy,\dotsc,f_{d,r_d}^\yy)$ for each $3 \leq d \leq D$, $\hat \f_i^\yy := (\f_{m-i+2}^\yy,\dotsc,\f_D^\yy)$ for each $0 \leq i \leq m-1$ and $\tilde \f^\yy := \hat \f_3^\yy$. Similarly, we define new weight functions $W_\yy$ by
\[
W_\yy(\ttt) := W(\ttt)W(\ttt +B^{-1} p_m \yy).
\]
By Proposition \ref{prop:weights}, we have $W_\yy \in \cC(R,(\kappa_j'))$, where $\kappa_j'$ may be bounded in terms of the numbers $\kappa_0,\dotsc,\kappa_j$. In this notation, we conclude from the above considerations that
\begin{equation}
\label{eq:sigma_1'}
\Sigma \leq \sum_{\yy \in \ZZ^n} N_{W_\yy}(\tilde \f^\yy,B,\tilde\q) -  p_m^{-n} \tilde Q^{-1} \cB^2.
\end{equation}
In fact, by the definition of $W_\yy$, the sum runs only over $\yy \ll B/p_m$, a fact we shall soon use.

If we denote the leading form of each polynomial $f_i^\yy$ by $F_i^\yy$, then we observe that 
\[
F_{i}^{\mathbf{y}}=p_m\mathbf{y}\cdot\nabla F_{i},
\]
unless the right hand side vanishes identically in $\mathbf{x}$. But by Remark \ref{rem:degenerate}, this happens only if $\yy = \0$. In particular, for $\yy \neq \0$, $f_i^\yy$ is a polynomial of degree $d_i -1$. We shall shortly invoke our induction hypothesis, with $\f$ replaced by $\tilde\f^\yy$ and $m$ replaced by $m-1$, to obtain an estimate for
\begin{equation}
\label{eq:V_y}
V_{W_\yy}(\tilde \f^\yy,B,\tilde \q) = N_{W_\yy}(\tilde \f^\yy,B,\tilde\q) - \tilde Q^{-1} \sum_{\xx \in \ZZ^n} W_\yy(B^{-1}\xx)
\end{equation}
for each $\yy \neq \0$. 
We then need to show that the last term in \eqref{eq:V_y}, when summed over $\yy \in \ZZ^n$, approximately cancels the last term in \eqref{eq:sigma'}. To this end, note that
\begin{multline*}
\sum_{\yy \in \ZZ^n} N_{W_\yy}(\tilde\f^\yy,B,\tilde\q) - \sum_{\yy \in \ZZ^n} V_{W_\yy}(\tilde\f^\yy,B,\tilde\q) \\
= \tilde Q^{-1} \sum_{\yy \in \ZZ^n} \sum_{\xx \in \ZZ^n}  W(B^{-1}(\xx+p_m \yy)) W(B^{-1}\xx),
\end{multline*}
which by Lemma \ref{lem:poisson} is
\[
= p_m^{-n} \tilde Q^{-1} \cB^2 + O_N\left(p_m^{-n}\tilde Q^{-1} B^{2n} (p_m/B)^{N}\right)
\]
for any $N \in \NN$. The rightmost term here is
\[
\ll p_m^{-n+r} B^{2n} \xi^{-\kappa} \left(\frac{\xi}{B}\right)^{N},
\]
and thus its contribution to $S$ is
\begin{align*}
&\ll B^n \xi^{-\kappa/2} \left(\frac{\xi}{B}\right)^{N} \leq B^n \xi^{-\cR} \left(\frac{\xi}{B}\right)^{N},
\end{align*}
since $\kappa \geq 2\cR$. This is certainly admissible if $N$ is chosen large enough. Thus we obtain
\begin{equation}
\label{eq:induction}
\begin{split}
V_W(\f,B,\q) &\ll p_m^{(n-r)/2} \left(\sum_{\yy \ll B/p_m} V_{W_\yy}(\tilde\f^\yy,B,\tilde\q) \right)^{1/2} \\
&\hphantom{\ll} + O\left(B^n \xi^{-\cR} \left(\frac{\xi}{B}\right)^{(n-1)/2^m}\right).
\end{split}
\end{equation}

For each $\yy \in \ZZ^n$, we define
\[
\rho(\yy) := \min_{0 \leq i \leq m-1} \rho_{p_i}(\tilde \f^\yy). 
\]
The number of equations in our new system $\tilde \f^\yy = \0$ is $\tilde r:= r-r_2$, so we have $0 \leq \rho(\yy) \leq \tilde r$. In estimating the sum
\begin{equation}
\label{eq:sum_V_y}
\sum_{\yy \ll B/p_m} V_{W_\yy}(\tilde\f^\yy,B,\tilde\q), 
\end{equation}
we shall first consider the contribution from those $\yy$ for which the system $\tilde\f^\yy$ still defines a complete intersection, that is, for which $\rho(\yy) = \tilde r$. In this case our induction hypothesis applies. Putting
\[
s(\yy)  := \max_{0 \leq i \leq m-1} s_{p_i}(\hat \f_i^\yy),
\]
we have
\[
V_{W_\yy}(\tilde\f^\yy,B,\tilde\q) \ll B^n \xi^{-\tilde \cR} \left(\frac{\xi}{B}\right)^{(n-s(\yy)-2)/2^{m-1}} 
\]
where
\[
\tilde \cR = \sum_{i=2}^{m} (1-2^{-i+1})r_{i+1} + \sum_{i=m+1}^{D-1} r_{i+1}. 
\]
We shall apply Lemma \ref{lem:HBlemma2} to estimate how often $s(\yy)$ attains a given value. Thus, suppose that $s(\yy)  =t$, where $-1 \leq t \leq n-1$. (Here we ignore the fact that our assumption on $\rho(\yy)$ would allow us to use a smaller upper bound for $s(\yy)$, since the bounds that we use are such that large values of $t$ make a negligible contribution.) Let us fix a prime $p_i$ such that $s_{p_i}(\hat \f_i^\yy) = t$. Invoking our condition \eqref{eq:primes_gcd} and recalling that $s_{p_i}(\hat \f_i) = -1$ by assumption, we shall then apply Lemma \ref{lem:HBlemma2} to the forms $F_{d,j}$, where $i+2 \leq d \leq D$.  If we denote by $V_\yy \subset \PP^{n-1}_{\FF_{p_i}}$ the closed subvariety defined by
\[
F_{d,j}^\yy = 0, \quad 1\leq j \leq r_d,\ i+2\leq d \leq D,
\]
the above assumption means that $\dim (\Sing(V_\yy))= t$. However, since $\rho_{p_i}(\tilde \f^\yy) = \tilde r$ by assumption, and thus necessarily $\rho_{p_i}(\hat \f_i^\yy) = r_{i+2}+\dotsb+r_D$, the singular locus of $V_\yy$ is precisely the algebraic set $S_\yy$ defined in Lemma \ref{lem:HBlemma2}, by the Jacobian Criterion. Hence, $\yy$ lies in (the affine cone over) the algebraic set $T_s$. By Lemma \ref{lem:HBlemma2} and Lemma \ref{lem:trivial}, we conclude that there are at most $O((B/p_m)^{n-t-1})$ choices of $\yy$ for which $s(\yy) = t$.
Thus, the contribution to \eqref{eq:sum_V_y} from $\yy$ such that $\rho(\yy) = \tilde r$ is at most
\begin{align}
\nonumber
&\ll B^n \xi^{-\tilde \cR} \left(\frac{\xi}{B}\right)^{(n-1)/2^{m-1}} \sum_{t'=0}^{n} \sum_{\substack{\yy \ll B/p_m\\s(\yy) = t'-1}} \left(\frac{B}{\xi}\right)^{t'/2^{m-1}} \\
\nonumber
&\ll B^n \xi^{-\tilde \cR} \left(\frac{\xi}{B}\right)^{(n-1)/2^{m-1}} \sum_{t'=0}^{n} \left(\frac{B}{\xi}\right)^{n-t'(1-1/2^{m-1})} \\
\nonumber
&\ll B^{2n} \xi^{-n-\tilde \cR} \left(\frac{\xi}{B}\right)^{(n-1)/2^{m-1}}.
\end{align}
The contribution to the right hand side of \eqref{eq:induction} from these values of $\yy$ is then
\begin{align*}
&\ll p_m^{(n-r)/2}  B^{n} \xi^{-(n+\tilde \cR)/2} \left(\frac{\xi}{B}\right)^{(n-1)/2^{m}} \\
&\ll B^n \xi^{-(r+\tilde \cR)/2} \left(\frac{\xi}{B}\right)^{(n-1)/2^{m}}\\
& = B^n \xi^{-\cR} \left(\frac{\xi}{B}\right)^{(n-1)/2^{m}}
\end{align*}
as required.

Now we treat those values of $\yy$ for which $\rho(\yy) < \tilde r$. In this case, the induction hypothesis does not apply, but we may apply a more elementary bound. If $\rho(\yy) = \rho$, then there is some $p_i$ such that $\rho_{p_i}(\tilde \f^\yy) = \rho$, so a crude upper bound for $N_{W_\yy}(\tilde \f^\yy,B,\tilde\q)$, and hence for $V_{W_\yy}(\tilde \f^\yy,B,\tilde \q)$, is given by $B^n\xi^{-\rho}$, by Lemma \ref{lem:trivial}. Furthermore, it is easy to see \cite[Lemma 4.2]{Marmon08} that in this case one actually has $\yy \in T_{n-\rho-1}$. Again, by Lemma \ref{lem:HBlemma2} and Lemma \ref{lem:trivial} we conclude that the contribution to \eqref{eq:sum_V_y} from values of $\yy$ with $\rho(\yy) < \tilde r$ (including $\yy = \0$) is at most
\begin{align}
\nonumber
&\ll \sum_{\rho = 0}^{\tilde r -1} \sum_{\substack{\yy \ll B/p_m\\ \rho(\yy) = \rho}} B^n \xi^{-\rho} \ll \sum_{\rho = 0}^{\tilde r -1} \left(\frac{B}{\xi}\right)^{\rho} B^n \xi^{-\rho} \\
\nonumber
&= B^n \sum_{\rho = 0}^{\tilde r -1}\left(\frac{B}{\xi^2}\right)^{\rho} \ll B^n.
\end{align}
This yields a contribution to the right hand side of \eqref{eq:induction} of size at most
\[
\ll B^n \xi^{-r/2} \left(\frac{\xi}{B} \right)^{n/2}, 
\]
corresponding to the second error term in \eqref{eq:asymptotic'}. This concludes the treatment of the case $s = -1$.

Next, assume that $s \geq 0$. By Lemma \ref{lem:hyperplane}, there is a primitive integer vector $\aa \ll_{n,D} 1$ such that
\begin{equation}
\label{eq:slice'}
\rho_{p_i}(\f \rvert_\aa) = r+1 \quad \text{and} \quad s_{p_i}(\hat\f_i\rvert_\aa) = s-1
\end{equation}
for all $0\leq i \leq m$, where $\f\rvert_\aa := (f_1,\dotsc,f_r,a_1 x_1+\dotsb a_n x_n)$, and analogously for $\hat\f_i\rvert_\aa$. We may then find a unimodular matrix $M \in \mathrm{GL}_n(\ZZ)$, all of whose entries are $O_{n,D}(1)$, such that $M^T \aa = \ee_n$, and write
\begin{align*}
\sum_{\substack{\xx \in \ZZ^n\\  \q \mid \f(\xx)}} W(B^{-1}(\xx)) &= \sum_{b \ll B} \sum_{\substack{\xx \in \ZZ^n\\ \q\mid \f(\xx),\ \aa\cdot \xx = b}} W(B^{-1}\xx)\\
&= \sum_{b \ll B} \sum_{\substack{\xx' \in \ZZ^n\\  \q \mid \f(M\xx'),\ x'_n = b}} W(B^{-1}M\xx') \\
&= \sum_{b \ll B} \sum_{\substack{\xx \in \ZZ^{n-1}\\ \q \mid {\g}_b(\xx)}} \tilde{W}_{B^{-1}b} (B^{-1}\xx).
\end{align*}
Here we have put ${\g}_b = (g_{b,1},\dotsc,g_{b,r})$, where $g_{b,i}$, for each $1\leq i \leq r$, is the polynomial in $n-1$ variables given by $g_{b,i}(\xx) := f_i(M(\xx,b))$, and we have defined $\tilde W_u : \RR^{n-1} \to \RR$ by $\tilde W_u(\ttt) := W(M(\ttt,u))$. Moreover, we have
\[
\sum_{b \ll B} \sum_{\xx \in \ZZ^{n-1}} \tilde{W}_{B^{-1}b} (B^{-1}\xx) = \sum_{\xx \in \ZZ^n} W(B^{-1}M\xx) = \sum_{\xx \in \ZZ^n} W(B^{-1}\xx),
\]
as $M$ is unimodular.

It follows that
\begin{align}
\nonumber V_W(\f,B,\q) &=  \sum_{\substack{\xx \in \ZZ^n\\ \q\mid \f(\xx)}} W(B^{-1}(\xx)) - Q^{-1} \sum_{\xx \in \ZZ^n} W(B^{-1}(\xx)) \\
\label{eq:slicesum'}
&= \sum_{b \ll B} V_{\tilde W_{B^{-1}b}}\big(\g_b,B,\q\big).
\end{align}

Note that the polynomials $\g_{b,i}$ have the same leading forms $G_i(\xx) := F_i(M(\xx,0))$ for each $b \in \ZZ$ (indeed, by \eqref{eq:slice'}, $G_i$ cannot vanish identically), and that
the subscheme $V(G_1,\dotsc,G_r) \subset \PP^{n-2}_\ZZ$ is isomorphic to the subscheme $V(F_1,\dotsc,F_r,\aa \cdot \xx) \subset \PP^{n-1}_\ZZ$. Thus we may conclude by \eqref{eq:slice'} that $\rho_{p_i}(\g_b) = r$ and $s_{p_i}(\g_b) = s-1$ for all $b \in \ZZ$ and all $0 \leq i \leq m$. Moreover, we have $\tilde W_{B^{-1}b} \in \cC(R',(\kappa_j'))$ where $R'\ll R$ and $\kappa_j' \ll \kappa_j$, by Proposition \ref{prop:weights}. Consequently, we may invoke our induction hypothesis to obtain the bound
\begin{multline*}
V_{\tilde W_{B^{-1}b}}\big(\g_b,B,\q\big) \ll B^{n-1} \xi^{-\cR} \left(\frac{\xi}{B}\right)^{(n-s-2)/2^m} 
\\
+ B^{n-1} \xi^{-r/2} \left(\frac{\xi}{B} \right)^{(n-s-1)/2}.
\end{multline*}
Summing over $b \ll B$ according to \eqref{eq:slicesum'}, we get
\[
V_W(\f,B,\q) \ll  B^{n} \xi^{-\cR} \left(\frac{\xi}{B}\right)^{(n-s-2)/2^m} 
+ B^{n} \xi^{-r/2} \left(\frac{\xi}{B} \right)^{(n-s-1)/2},
\]
as desired.

\section{Proof of the main results}
\label{sec:proof_main}

Suppose that we are given a system $\f$ of polynomials with $\rho_\infty(\f) = r$ and $s_\infty(\f) = -1$. We want to apply Proposition \ref{prop:heart}, so our first step will be to find an alternative system $\g = (g_1,\dotsc,g_r)$, generating the same ideal $\langle\f\rangle$, such that $s_\infty(\hat \g_i) = -1$ for each $i=0,\dotsc,D-2$. The following result, which is a slight modification of \cite[Lemma 3.1]{Browning-Dietmann-Heath-Brown}, shows that this is possible.

\begin{lemma}
\label{lem:optimalsystem}
Let $\f=(f_1,\dotsc,f_r)$ be a system of polynomials in $\ZZ[\xx]$, with 
\[
\deg(f_1) \leq \dotsb \leq \deg(f_r), 
\]
and let $I = \langle\f\rangle$ be the ideal in $\ZZ[\xx]$ that it generates. Suppose that the leading forms $F_1,\dotsc,F_r$ define a non-singular complete intersection of codimension $r$ in $\PP^{n-1}_\QQ$. Then there is another system $\g = (g_1,\dotsc,g_r)$, with $\deg(g_i) = \deg(f_i)$ for all $i$, 
such that the corresponding leading forms $G_1,\dotsc,G_r$ have the following property:
\begin{equation}
\label{eq:optimalsystem}
\begin{split}
&\text{For each $0 \leq j \leq r$, the subset $G_{r-j+1},\dotsc,G_r$ defines} \\
&\text{a non-singular complete intersection of codimension $j$.}
\end{split}
\end{equation}
Furthermore, $\Vert\g\Vert$ may be bounded above by a polynomial in $\Vert\f\Vert$ whose height and degree are bounded in terms of $n$ and $\d$.
\end{lemma}

\begin{proof}
Let us write $I'$ for the ideal generated by the leading forms $F_i$. In the proof of \cite[Lemma 3.1]{Browning-Dietmann-Heath-Brown}, a new system of generators $G_i$ satisfying the property \eqref{eq:optimalsystem} is found by taking
\begin{equation}
\label{eq:pencil}
G_i = \sum_{\substack{1\leq j \leq i\\ d_i = d_j}}\lambda_{ij} F_j + 
\sum_{\substack{1\leq j \leq i\\ d_i > d_j}} \sum_{1 \leq k \leq n} \lambda_{ijk} x_k^{d_i-d_j} F_j 
\end{equation}
for suitable integers $\lambda_{ij},\lambda_{ijk}$. It is not hard to see that one may take $\lambda_{ii} = 1$ for all $i$, so that the $G_i$ indeed generate the ideal $I'$ as a $\ZZ[\xx]$-module. It follows immediately by construction that the polynomials
\[
g_i = \sum_{\substack{1\leq j \leq i\\ d_i = d_j}}\lambda_{ij} f_j + 
\sum_{\substack{1\leq j \leq i\\ d_i > d_j}} \sum_{1 \leq k \leq n} \lambda_{ijk} x_k^{d_i-d_j} f_j 
\]
then generate the ideal $I$. 

It remains to verify the height bound. To this end, we need to redo the above argument in a more abstract setting. Let $\PP_{d_i}$, for $i = 1,\dotsc,r$ be the projective space $\PP_\ZZ^{N_i}$ parameterizing degree $d_i$ hypersurfaces in $\PP^{n-1}_\ZZ$. By abuse of notation, we will write $F_i$ for the element of $\PP_{d_i}$ representing the hypersurface $F_i=0$, and $\F=(F_1,\dotsc,F_r)$ for an element of $\PP_{d_1}\times \dotsb \times \PP_{d_r}$. Moreover, for each $i=1,\dotsc,r$, we consider parameter spaces $\mathbf{\Lambda}_i = \PP^{M_i-1}_\ZZ$, where $M_i$ is the number of unknowns $\lambda_{ij},\lambda_{ijk}$ occurring in the right hand side of \eqref{eq:pencil} for a fixed $i$, and denote a general element of $\mathbf{\Lambda}_i$ by $\llambda_i$. Let $V_{\F,\llambda_i}$ denote the hypersurface in $\PP^{n-1}$ defined by the form $G_i$ in \eqref{eq:pencil}. It is then a standard fact that the locus of $(\F,\mathbf{\lambda}_1,\dotsc,\mathbf{\lambda}_r)$ such that the intersection $V_{\F,\llambda_r} \cap \dotsb \cap V_{\F,\llambda_{r-j+1}}$ is proper and smooth for each $j=1,\dotsc,r$ is an open subset $\mathcal{U}$ of 
$\PP_{d_1}\times \dotsc \times \PP_{d_r} \times \LLambda_1\times \dotsb \LLambda_r$. Its complement, being a closed subset, is thus defined by the simultaneous vanishing of a finite  collection of multihomogeneous forms $\cF_i(\F;\llambda_1;\dotsc;\llambda_r)$ with coefficients in $\ZZ$. 

For the specific $r$-tuple $\F$ under consideration, we have already verified the non-emptiness of $\pi^{-1}(\F)$, where $\pi$ denotes the projection from $\mathcal{U}$ to $\PP_{d_1}\times \dotsc \times \PP_{d_r}$. In particular, at least one of these forms $\cG(\llambda_1;\dotsc;\llambda_r):=\cF_i(\F;\llambda_1;\dotsc;\llambda_r)$, say, does not vanish identically in $\llambda_1,\dotsc,\llambda_r$. Clearly, we have bounds 
\[
\Vert \cG \Vert\ll_{n,\d} \Vert \F \Vert^{O_{n,\d}(1)} \quad \text{and} \quad  \deg(\cG) \ll_{n,\d} 1.
\]
As an easy consequence of Lemma \ref{lem:trivial}, we may then find a solution $\llambda_1,\dotsc,\llambda_r$ to $\cG(\llambda_1;\dotsc;\llambda_r)\neq 0$ satisfying $|\llambda_i| \ll_{n,\d} \Vert \F \Vert^{O_{n,\d}(1)}$. This in turn gives the desired bound on $\Vert\g\Vert$.
\end{proof}

By Lemma \ref{lem:optimalsystem} we may now assume that $s_\infty(\hat{\f_i}) = -1$ for each $i=0,\dotsc,D-2$. Indeed, the possible increase in $\Vert \f \Vert$ is absorbed by the implied constant in \eqref{eq:different}. 

\begin{lemma}
\label{lem:chooseprimes}
Suppose that
\[
\rho_\infty(\f) = r \text{ and } s_{\infty}(\hat{\f}_0) = \dotsb = s_{\infty}(\hat{\f}_{D-2}) = -1.
\]
Then, provided $\xi \gg \log \Vert \F \Vert$, there exist $m$ distinct primes 
\(
p_1,\dotsc,p_m
\)
with $p_j \asymp \xi$, such that  
\begin{equation}
\label{eq:goodprimes}
\rho_{p_j}(\f) = r \text{ and } s_{p_j}(\hat{\f}_0) = \dotsb = s_{p_j}(\hat{\f}_{D-2}) = -1  
\end{equation}
for all $j=1,\dotsc,m$. The implied constants depend only on $n$, $\d$ and $m$.
\end{lemma}

\begin{proof}
By an argument similar to the one in the proof of Lemma \ref{lem:optimalsystem}, the set of $\F=(F_1,\dotsc,F_r)$ such that either $\rho_\infty(\F) < r$ or $s_\infty(\hat{\F}_i) > -1$ for some $i$ is a closed subset of $\PP_{d_1}\times \dotsb \times \PP_{d_r}$, defined by a collection of multihomogeneous forms. By assumption, at least one of these, say $\cG(\F)$, does not vanish when $\F$ is taken to be the system of leading forms of the system $\f$ in the hypotheses. Furthermore, any prime $p$ violating condition \eqref{eq:goodprimes} has to divide the integer $\cG(\F)$.

It is a straightforward consequence of Chebyshev's theorem that for any natural number $m$, there is a constant $C$ depending only on $m$ such that for any natural number $A\geq 3$, the interval $(\log A,C \log A]$ contains at least $m$ distinct primes not dividing $A$. From this the lemma follows by taking $A = \cG(\F)$ and observing that 
\[
\log \cG(\F) \ll_{n,\d} \log \Vert F \Vert.  
\]
\end{proof}

We shall derive the estimate in Theorem \ref{thm:different} by applying Proposition \ref{prop:heart} with $m = D-2$. A suitable smooth weight function is given by
\begin{equation*}
 W(\mathbf t) = \prod_{i=1}^n w(t_i/2), \text{ where } 
 w(t)=
 \begin{cases}
  \exp(-1/(1-t^2)), & \vert t \vert < 1,\\
  0, & \vert t \vert \geq 1.
 \end{cases}
\end{equation*}
The asymptotic formula \eqref{eq:asymptotic'} in particular yields an upper bound for $N_W(\f,B,q)$. Assuming that $s=-1$, the optimal such upper bound will be obtained when
\begin{equation}
\label{eq:xi}
\xi = B^{1-\eta}, 
\end{equation} 
where $\eta$ is the quantity defined in Theorem \ref{thm:different}. Provided that $n > 2^{D-2}\Delta$, this implies that $\xi \geq B^{1/2}$, as required in Proposition \ref{prop:heart}. Suppose first that 
\[
B \gg (\log\Vert \F \Vert)^{1/(1-\eta)}.
\]
In view of Lemma \ref{lem:chooseprimes}, we may then indeed apply Proposition \ref{prop:heart}, and the primes $p_0,\dotsc,p_{D-2}$ may be chosen so that $s=-1$. We obtain
\[
N(\f,B) \ll B^n \xi^{-\D'} + B^n \xi^{-\cR} \left(\frac{\xi}{B}\right)^{(n-1)/2^{D-2}} + B^n \xi^{-r/2} \left(\frac{\xi}{B} \right)^{n/2}.
\]
In \eqref{eq:xi}, we have already set $\xi$ to make the first two terms on the right hand side equal. Verifying that the third term is then of negligible size as soon as $D \geq 4$, we arrive at the bound
\[
N(\f,B) \ll B^{n-\D'(1-\eta)},
\]
thus establishing the bound in Theorem \ref{thm:different} in this case. Here, the implied constant depends only on $n$ and $\d$.

In the complementary case where $B \ll (\log\Vert \F \Vert)^{1/(1-\eta)}$, we may for example use the trivial bound from Lemma \ref{lem:trivial} to obtain
\begin{align*}
N(\f,B) &\ll B^{n-r} \ll B^{n-\D'(1-\eta)} (\log\Vert \F \Vert)^{\cD'-r/(1-\eta)} \\
&\leq B^{n-\D'(1-\eta)} (\log\Vert \F \Vert)^{\cD'}, 
\end{align*}
as desired. This completes the proof of Theorem \ref{thm:different}.

To prove Theorem \ref{thm:same}, one instead takes $m = d -2 \geq 2$ in the application of Proposition \ref{prop:heart}. More generally, for certain given specifications of the data $n$, $\d$, the optimal upper bound for $N(\f,B)$ may be attained by choosing some other value of $m \in \{0,\dotsc,D-2\}$. Thus, there is a whole range of intermediate results between Theorem \ref{thm:same} and Theorem \ref{thm:different}, the precise formulation of which is left to the reader.

\section{Concluding remarks}

Continuing the discussion from Remark \ref{rem:improvement}, it would be desirable to have a result where the exponent of $B$ approaches the heuristic one also for systems of truly differing degree. In fact, the expression $n -\cD'$, which is the asymptotic value of the exponent of $B$ in Theorem \ref{thm:different}, as $n \to \infty$, would be the expected one for a system where all polynomials of degree $d \leq D-1$ were replaced by polynomials of degree $d-1$.  This is because our setup at the moment needs to exclude polynomials of degree $\leq 2$ in each differencing step. 
 
One could also imagine using ideas from \cite{Marmon10} to improve our results. That paper considers the case where $r = 1$ and $m=2$. Using ideas due to Salberger \cite{Salberger-Integral_points}, an alternative differencing procedure is used, which allows one to keep the original polynomial in play in each step (see Remark \ref{rem:salberger}). This produces varieties of gradually increasing codimension, giving stronger bounds. But it is not obvious how to turn the arguments in \cite{Marmon10} into an iterative approach, nor do the geometric considerations playing the role of Lemma \ref{lem:HBlemma2} in \cite{Marmon10} generalize readily to the case $r > 1$. 

\subsection*{Acknowledgements}
I wish to thank the anonymous referee for valuable remarks, and J\"org Br\"udern for helpful discussions concerning this paper.

\bibliographystyle{plain}
\bibliography{ratpoints}

\end{document}